\documentclass[a4paper,12pt]{article}

\RequirePackage{amsthm,amsmath,amsfonts,amssymb}
\RequirePackage[numbers]{natbib}
\RequirePackage{graphicx}

\theoremstyle{plain}
\newtheorem{theorem}{Theorem}[section]
\newtheorem{proposition}[theorem]{Proposition}
\newtheorem{lemma}[theorem]{Lemma}
\newtheorem{corollary}[theorem]{Corollary}

\theoremstyle{definition}
\newtheorem{definition}[theorem]{Definition}
\newtheorem{example}[theorem]{Example}
\newtheorem{remark}[theorem]{Remark}
\newtheorem{algorithm}{Algorithm}

\begin{document}

\begin{center}

  \Large
      {\bf Direct Sampling from Conditional Distributions
        by Sequential Maximum Likelihood Estimations}
  \normalsize

  \bigskip By \bigskip

  \textsc{Shuhei Mano}

  \smallskip
  
  The Institute of Statistical Mathematics, Japan


\end{center}

\small

{\bf Abstract.}
We can directly sample from the conditional distribution of any
log-affine model. The algorithm is a Markov chain on a bounded
integer lattice, and its transition probability is the ratio of
the UMVUE (uniformly minimum variance unbiased estimator) of
the expected counts to the total number of counts.
The computation of the UMVUE accounts for most of
the computational cost, which makes the implementation challenging.
Here, we investigated an approximate algorithm that replaces
the UMVUE with the MLE (maximum likelihood estimator).
Although it is generally not exact, it is efficient and easy to
implement; no prior study is required, such as about the connection
matrices of the holonomic ideal in the original algorithm.

\smallskip


\smallskip


\normalsize

\section{Introduction}
\label{sect:intr}

Consider a discrete sample with state space
$[m]:=\{1,2,\ldots,m\}$ for a positive integer $m$.
In this paper, we will discuss the {\it discrete exponential
families} in statistics, which are also called log-affine
models or toric models.

\begin{definition}\label{defi:toric}
  Let $A=(a_{ij})$ $\in \mathbb{Z}^{d\times m}$ be
  a matrix of integers such that no row or column is
  the zero vector and $(1,\ldots,1)\in {\rm rowspan}(A)$.
  Let $x\in\mathbb{R}^m_{>0}$. The {\it log-affine model}
  associated with the {\it configuration matrix} $A$
  is the set of probability distributions 
  \[
    \mathcal{M}_A:={\rm cl}\{p\in{\rm int}(\Delta_{m-1}):
    \log p \in \log x+{\rm rowspan}(A)\},
  \]
  where $\Delta_{m-1}$ is the standard $m$-dimensional simplex.
  If $x=1$, the model is called the {\it log-linear model.}  
\end{definition}

In this paper, we consider the closure, called the
extended log-affine model by \cite[Section~4.2.3]{Lau94},
whose support may be at the boundary of the simplex.
This extension ensures the existence of the maximum
likelihood estimate (MLE), whether or not observed counts
satisfy constraints to be positive.

The vector $p$ in Definition~\ref{defi:toric} may be
parameterized as
\begin{equation}
  p_j=\frac{x_j}{Z(t,x)}
  \prod_{i\in[d]}t_i^{a_{ij}}, \quad j \in [m]
  \label{para}
\end{equation}
with unknown parameters $t\in\mathbb{R}_{>0}^d$ and
the normalization constant $Z(t,x)$ such that
$\sum_{j\in[m]} p_j=1$. A sample consisting of
observations in which counts of the $j$-th state is
$u_j$, $j\in[m]$ taken from the multinomial distribution
specified by the probability
mass function \eqref{para} follows the probability law
\begin{equation}\label{prob}
  \mathbf{P}(U=u,AU=b)
  =|u|!\prod_{j\in[m]}\frac{p_j^{u_j}}{u_j!}
  =\frac{|u|!}{\{Z(t,x)\}^{|u|}}
  \prod_{i\in[d]}t_i^{b_i}
  \prod_{j\in[m]}\frac{x_j^{u_j}}{u_j!},
\end{equation}
where the total number of counts is $|u|:=u_1+\cdots+u_m$.
The conditional distribution given by complete minimal
sufficient statistics $b$ in the affine semigroup
$\mathbb{N}A:=\{Av:v\in\mathbb{N}^m\}$, where $\mathbb{N}$
is the set of non-negative integers, is 
\begin{equation}
  \mathbf{P}(U=u|AU=b)=\frac{1}{Z_A(b;x)}\frac{x^u}{u!},
  \quad x^u:=\prod_{j\in [m]}x_j^{u_j}, \quad
  u!:=\prod_{j\in[m]}u_j!.
  \label{cond}
\end{equation}
The support of this distribution is determined by
the linear transformation
\begin{equation}
  \mathcal{F}_A(b):=\{v\in\mathbb{N}^m:Av=b\}.
  \label{fiber}
\end{equation}
The set \eqref{fiber} is called the $b$-{\it fiber}
associated to the configuration matrix $A$ with
the sufficient statistics $b$. The normalization constant,
or the partition function
\begin{equation}
  Z_A(b;x):=\sum_{v\in \mathcal{F}_A(b)}\frac{x^v}{v!}
  \label{Apol}
\end{equation}
is called the $A$-{\it hypergeometric polynomial,} or
the {\it GKZ-hypergeometric polynomial,} defined by Gel'fand,
Kapranov, and Zelevinsky in the 1980's. Note that the condition
$(1,\ldots,1)\in{\rm rowspan}(A)$ in Definition~\ref{defi:toric}
means that this polynomial is homogeneous.
We adopt the convention $Z_A(b;x)=0$ if $b\notin\mathbb{N}A$.
It is straightforward to see that a sample consisting of
the Poisson random variables with means $\lambda_j$, $j\in[m]$
with $p_j=\lambda_j/|\lambda|$ also follows the conditional
distribution \eqref{cond}.

Since the discrete exponential family is the standard family
of discrete distributions, sampling from the conditional
distribution \eqref{cond} has various statistical demands.
A typical use in statistics is hypothesis testing.
To evaluate the significance of our data under a hypothetical
model, we generate random variables called statistics from
its distribution and estimate the significance by
the proportion of generated values greater than the observed
value. When evaluating power of a test, we generate random
variables from a distribution under a certain hypothetical
model and estimate the power by the proportion of generated
values that exceed a given significance level. Such
hypothetical models include the non-independence model of
two-way contingency tables and the no-three-way interaction
model for the three-way contingency tables. However,
sampling from such models has also been considered to be
usually impossible.
(see, e.g. \cite[Chapter 9]{Sul18}). This motivates us
to use the Metropolis algorithm, originally
devised for sampling from canonical distributions in
physics and one of the most common tools in the Markov
chain Monte Carlo methods. In the Metropolis algorithm,
the unique stationary distribution of an ergodic
Markov chain is the distribution we need (see, e.g.,
\cite[Chapter 3]{LP17}). We call the chain
{\it Metropolis chain}. Diaconis and Sturmfels \cite{DS98}
proposed to construct a basis of moves in
the Metropolis chain by using Gr\"obner bases of the toric
ideal of a configuration matrix. They called such a basis, or
a generating set for the toric ideal, a {\it Markov basis.}
The state space of the chain is the $b$-fiber \eqref{fiber},
and the ratio of two monomials, which correspond to
the current state and the proposed next state, is computed
in each step of the chain.

However, the Metropolis algorithm has unavoidable
drawbacks in principle. It cannot give independent
samples, and we must wait for the convergence of
the Metropolis chain to the stationary distribution,
whose assessment is a difficult problem to answer.
Therefore, it would be desirable, if possible, to
obtain independent samples directly from the target
distribution.

Importance sampling is an algorithm that draws samples
that follow an easily generated distribution, and by
weighting them based on the ratio to the target distribution,
we regard them as samples drawn from the target
distribution. Importance sampling can obtain independent
samples and has no issue of convergence. Chen et al.~\cite{CDS06}
proposed a sequential importance sampling algorithm
that draws contingency tables sequentially by each cell
by introducing the notion that they called the Markov
subbase. In practice, constructing the distribution to
sample from is non-trivial and requires solving an integer
or linear programming problem for each step.

Moreover, it remains true that importance sampling
is not a direct sampling from the target distribution.
The performance of an importance sampling is greatly
affected by the choice of the weights and
the distribution to sample from. Naturally, the optimal
case is when weights are not necessary when sampled
from the target distribution. Hence, a fundamental
problem remained: is it possible to sample directly
from the target distribution?

The answer is yes; we can sample from the conditional
distribution \eqref{cond} exactly using the holonomic
ideal generated by a system of linear partial differential
equations that the $A$-hypergeometric polynomial satisfies
\cite{Man17} (see Definition~\ref{defi:holo}).
The realizability is guaranteed by Gr\"obner bases of
the holonomic ideal generated by the $A$-hypergeometric
system. Throughout this paper, we call the sampling
algorithm proposed in \cite{Man17}
the {\it direct sampling algorithm}. See the monograph
\cite{Man18} for an extensive discussion.

The direct sampling algorithm is a random walk or a Markov
chain on a bounded integer lattice defined below. One
sample path of the Markov chain is equivalent to one
contingency table. See Example~\ref{exam:2x2:1}.
Note that while the sequential draw of the contingency table
is the same as in Chen et al.~\cite{CDS06}, as Markov
chains the two are quite different and the units of
the direct sampling are individual counts rather than cells.

\begin{definition}[\cite{MT21+}]\label{defi:lattice}
  Consider a matrix $A=(a_{ij})\in\mathbb{Z}^{d\times m}$
  satisfying the conditions in Definition~\ref{defi:toric} and
  a vector $b\in\mathbb{N}A$. The {\it Markov lattice}
  $\mathcal{L}_A(b)$ is the bounded integer lattice embedded
  in the affine semigroup $\mathbb{N}A$ equipped with
  the partial order
  \[
    \beta \in \mathbb{N}A~~\text{and}~~\beta-a_j
    \in \mathbb{N}A 
    \quad \Rightarrow \quad \beta-a_j \prec \beta,
  \]  
  and the maximum and the minimum are $b$ and $0$, respectively.
  Here, $a_j$ denotes the $j$-th column vector of $A$.
\end{definition}

\begin{example}[Two-way contingency tables of the independence model]\label{exam:2x2:1}
  We discuss the $2\times 2$ contingency table, but
  the same argument holds for general two-way
  contingency tables. The $2\times 2$ table is
  \[
  \begin{array}{cc|c}
    u_{11}&u_{12}&u_{1\cdot}\\
    u_{21}&u_{22}&u_{2\cdot}\\
    \hline
    u_{\cdot1}&u_{\cdot2}&|u|
  \end{array},
  \]
  where $u_{i\cdot}=u_{i1}+u_{i2}$ and $u_{\cdot j}=u_{1j}+u_{2j}$,
  $i,j\in[2]$. The configuration matrix, the vectors of
  the observed counts, and the vector of the sufficient
  statistics are
  \begin{equation*}
    A\!=\!\left(\begin{array}{cccc}
      1&1&0&0\\
      0&0&1&1\\
      1&0&1&0\\
      0&1&0&1
    \end{array}\right),\quad
    u\!=\!\left(
    \begin{array}{c}
      u_{11}\\u_{12}\\u_{21}\\u_{22}
    \end{array}
    \right),\quad {\rm and}\quad
    b=\left(
    \begin{array}{c}
      u_{1\cdot}\\u_{2\cdot}\\u_{\cdot1}\\u_{\cdot2}
    \end{array}
    \right),
  \end{equation*}
  respectively. The conditional distribution is
  \[
    \mathbf{P}(U=u|AU=b)
    =\frac{1}{Z_A(b;1)}\frac{1}{u_{11}!u_{12}!u_{21}!u_{22}!}.
  \]
Figure~\ref{fig1} is the Hasse diagram of an example.
\begin{figure}
  \begin{center}
\includegraphics[width=0.5\textwidth]{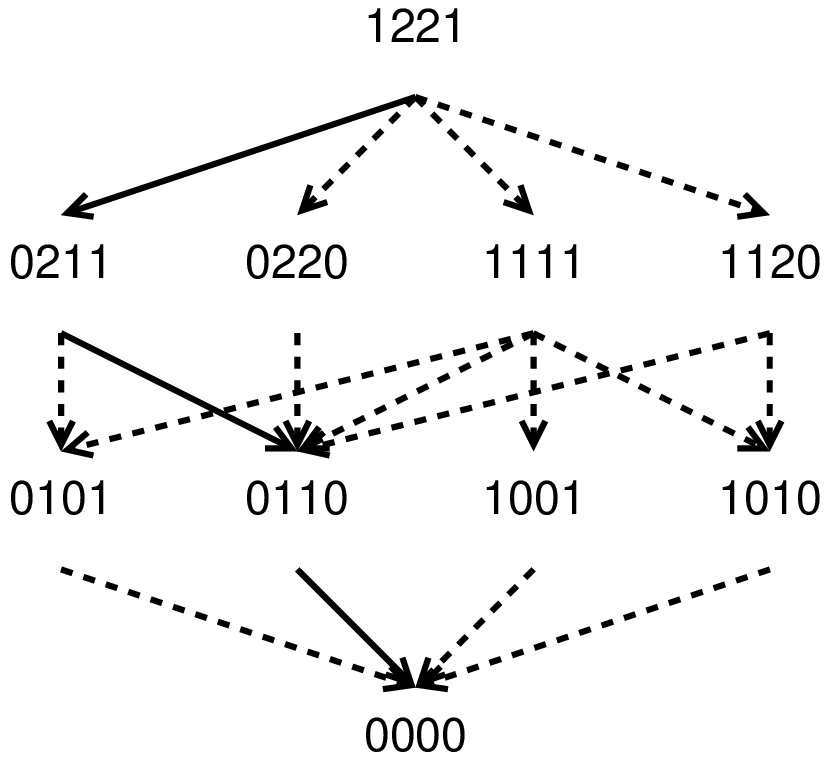}
\caption{The Markov lattice for the matrix
  in Example~\ref{exam:2x2:1} with the maximum $b=(1,2,2,1)^\top$.
  A sample path is shown by solid edges.} 
\label{fig1}
\end{center}
\end{figure}
Transitions occur along the edges. The sample path,
shown with solid edges, is equal to one contingency
table with sufficient statistics $b=(1,2,2,1)^\top$.
In the first line of the following tables, asterisks
mark the cells picked during each solid edge; each
time a cell is picked, the corresponding marginal
count drops by one. The second line shows how
the contingency table is created by listing these
picked counts. The general form of this process is
given in Algorithm~\ref{algo:0}.
\begin{align}
  \begin{array}{cc|c}
   *& &1\\
   & &2\\
  \hline  
  2& 1  
  \end{array}
  &~\quad\to&
  \begin{array}{cc|c}
   & &0\\
    &*&2\\
  \hline
  1& 1  
  \end{array}
  &~\quad\to&
  \begin{array}{cc|c}
   & &0\\
  *& &1\\
  \hline  
  1& 0  
  \end{array}
  &~\quad\to&
  \begin{array}{cc|c}
   & &0\\
    & &0\\
  \hline  
  0& 0  
  \end{array}\nonumber\\
  &
  \begin{array}{cc|c}
  1&0&1\\
  0&0&0\\
  \hline  
  1&0&   
  \end{array}& &
  \begin{array}{cc|c}
  1&0&1\\
  0&1&1\\
  \hline  
  1&1&   
  \end{array}& &
  \begin{array}{cc|c}
  1&0&1\\
  1&1&2\\
  \hline  
  2&1&   
  \end{array}&\label{tables}
\end{align}
\end{example}

As we will see in Section~\ref{sect:algo}, Markov transitions
occur between the adjacent elements of the Markov lattice
with the Markov kernel represented by the transition probability
\begin{equation}
  P(\beta,\beta-a_j;x)
  :=\frac{Z_A(\beta-a_j;x)}{Z_A(\beta;x)}\frac{x_j}{\deg(\beta)},
  \quad j\in[m] \label{kernel}
\end{equation}
for the transition from $\beta$ to $\beta-a_j$, where $\deg(\beta)$
is the degree of $Z_A(\beta;x)$. The computation of
the transition probabilities dominates the computational cost of
the direct sampling algorithm. Evaluation of \eqref{kernel}
following the definition of the $A$-hypergeometric polynomial
\eqref{Apol} is practically impossible since the number of terms
in the polynomial is the cardinality of the fiber \eqref{fiber}
that increases rapidly with the total number of counts.
In Section~4 of \cite{MT21+}, the authors reported some results
of numerical experiments of exact computations with rational
number arithmetic on a computer algebra system.
For two-way contingency tables of the independence model,
the direct sampling algorithm demanded a time of about one-tenth
time of the Metropolis algorithm when we count the effective
sample size of samples taken from the Metropolis chain.
On the other hand, for a model without independence,
the direct sampling algorithm demanded a thousand times
the time the Metropolis algorithm demanded. For the latter
case, although the computation is exact, the computational
burden would prohibit the practical use of the direct
sampling algorithm.

We explain the reason behind this contrast in
the computational costs. For the independence model,
the transition probability \eqref{kernel} reduces to the product
of the proportion of the row marginal count to
the total number of counts and the proportion of
the column marginal count to the total number of counts,
a natural consequence of the independence.
Therefore, \eqref{kernel} reduces to a rational number,
and the computation is immediate. This argument was extended
to decomposable graphical models in Section~3.2 of \cite{MT21+}.
On the other hand, we have to compute \eqref{kernel}
with some effort for a model without independence.
More specifically, we need to compute the
\textit{connection matrices} of the Pfaffian system
of the holonomic ideal (see Section~\ref{subsect:IPS}).
But unfortunately, the computation demands time
growing exponentially with the number of rows or columns.

More generally, the main obstacle common to using
the Markov basis and the direct sampling algorithm
for general models is the computational cost of
Gr\"obner bases. The Gr\"obner basis guarantees
the existence of a Markov basis, however, Buchberger's
algorithm for computing Gr\"obner bases is impractical
since it demands a huge computational cost
(see, e.g., \cite[Section 2.9]{CLO07}). Therefore,
in practice, we have to avoid using Gr\"obner bases.
The challenge is to obtain Markov bases or connection
matrices by some consideration using properties
of each model, rather than resorting to Buchberger's
algorithm. A book-length treatment of this issue for
the Markov bases is \cite{AHT12} and for the direct
sampling algorithm is \cite{Man18}.

In this paper, we show that the transition probability
\eqref{kernel} is the ratio of the {\it uniformly minimum
variance unbiased estimator} (UMVUE) of the expected count
to the total number of counts, where for a vector of
observed counts $u\in\mathbb{N}^m$, we call $\mu:=\mathbf{E}U$
the vector of {\it expected counts}. In the final section
of Chapter~1, entitled ``Exact Methods,''
Haberman~\cite{Hab74} argued that the UMVUE was impractical
for two main reasons: first, the UMVUE generally does not
lie in the model manifold $\mathcal{M}_A$ (the toric
variety defined by the matrix $A$); second, it is
computationally intensive. In response to this, in
the final paragraph of the chapter, she concluded,
``Since exact methods are usually impractical, approximate
methods must be tried. One such method, that of maximum
likelihood, provides estimation and test procedures which
involve simple computational methods and which have large
sample properties. This method is the subject of
the remainder of the monograph.'' In this context,
the significance of the UMVUE seems to have been overlooked,
and the discussion of the MLE has dominated the literature.

In line with this trend, special MLEs have been attracting
attention in recent years. Discrete statistical models for
which the MLE admits an expression of a rational function
of the observed counts are called {\it rational}, and they
were characterized by \cite{Huh14} and \cite{DMS21}.
Such an expression of the MLE is called
the {\it rational MLE}. The MLE of the expected count
of decomposable graphical models is given by a rational
function of the observed counts
(see, e.g., \cite[Section 4.4]{Lau94}), and it was shown
that an undirected discrete
graphical model is rational if and only if it is a decomposable
graphical model \cite{GMS06}. In this paper, we will see that
the UMVUE and the MLE of an undirected discrete
graphical model coincide if and only if it is a decomposable
graphical model.

We restate the observation above about the direct sampling
algorithm as follows: if a model is log-linear and rational,
then the computation of the transition probability
\eqref{kernel} is immediate, and the direct sampling algorithm
works efficiently and exactly. Otherwise, we may
replace the UMVUE with the MLE and run the algorithm
using existing methods to compute the MLE, such as
the iterative proportional scaling algorithm,
with the expense of bias of the MLE, that asymptotically
disappears (See Remark~\ref{rema:bias}).

This paper is structured as follows. In
Section~\ref{sect:algo}, we introduce the details of
the direct sampling algorithm proposed in \cite{Man17}
and show that the transition probability is the ratio
of the UMVUE of the expected count to the total number
of counts. Then, after summarizing the relevant properties
of the MLE, we show that the UMVUE and the MLE
of an undirected discrete graphical model
coincide if and only if is a decomposable graphical model
and the resulting MLE is rational. Lastly, we discuss
the use of the MLE and reveal the key how we avoid computation
of the connection matrices of the holonomic ideal.
Section~\ref{sect:exp} is devoted to numerical experiments.
Using several contingency tables as examples, we demonstrate
how the direct sampling algorithm works.
Section~\ref{sect:disc} summarizes the significance of
this paper.

To the best of author's knowledge, this is the first
paper to demonstrate direct sampling from
the conditional distribution of the discrete exponential
family in a feasible time, although it is generally
approximate.

\section{Algorithm}
\label{sect:algo}

The direct sampling algorithm discussed in this paper
is as follows.

\begin{algorithm}\cite{Man17}\label{algo:0}~
  \begin{itemize}
  \item[] Input: A matrix $A\in\mathbb{Z}^{d\times m}$
    satisfying the conditions in Definition~\ref{defi:toric},
    vectors $x\in\mathbb{R}_{>0}^m$, $b\in\mathbb{N}A$ with
    the total number of counts $n$.
  \item[] Output: A vector of counts $u\in\mathbb{N}^m$ following
    the conditional distribution \eqref{cond} given $b$.
  \end{itemize}
  \begin{itemize}
  \item [] Step 1: Set $\beta=b$.

    \smallskip
    
    {\bf For} $t=1,2,\ldots,n$ {\bf do}

  \item [] ~~ Step 2: Compute $P(\beta,\beta-a_j;x)$ in \eqref{kernel}
    for each $j\in[m]$.

  \item [] ~~ Step 3: Pick $j_t=j$ with probability
    $P(\beta,\beta-a_j;x)$, $j\in[m]$.

  \item [] ~~ Step 4: Set $\beta\leftarrow \beta-a_{j_t}$.

    \smallskip
    
    {\bf End for}

  \item [] Step 5: Output $u_j:=\#\{t\in[n]:j_t=j\}$, $j\in[m]$.
  \end{itemize}
\end{algorithm}

In Input $b$ and $n$ should be consistent:
since $(1,\ldots,1)\in{\rm rowspan}(A)$, we can always
have the vector $c\in\mathbb{Q}^d$ with respect to $A$
such that $n=\sum_{i\in[d]}c_ib_i$.

Our concern is how to achieve Step~2, since the cost of
this step dominates the computational cost of
Algorithm~\ref{algo:0}. In the paper \cite{Man17},
the author obtained the connection matrices of the holonomic
ideal for the case of $d=2$ and discussed computing
$A$-hypergeometric polynomials to appear in advance.
In the paper \cite{MT21+}, the authors discussed computing
$A$-hypergeometric polynomials needed along a sample path
of the Markov chain for general cases. Both implementations
give exact results, but the latter is general but not
efficient.

\begin{remark}\label{rema:hit}
A merit of Algorithm~\ref{algo:0} is that we do not have to
construct the Markov lattice $\mathcal{L}_A(b)$. A Markov
transition occurs only between adjacent elements because 
otherwise, $Z_A(\beta-a_j;x)=0$ since $\beta-a_j\not\in \mathbb{N}A$.
For the same reason, a sample path never hits outside of
$\mathcal{L}_A(b)$, or more precisely, the set
$({\rm Cone}(A)\cap\mathbb{Z}^m)\setminus\mathbb{N}A$,
where ${\rm Cone}(A):=\{Av:v\in\mathbb{R}^m_{\ge 0}\}$
is the polyhedral cone generated by the column vectors of $A$.
\end{remark}

\subsection{Markov transitions and the UMVUE}
\label{subsect:UMVUE}
  
We recall some basic properties of the $A$-hypergeometric
polynomial \eqref{Apol}. A book-length treatment of
the $A$-hypergeometric functions is \cite{SST00}.

\begin{definition}\label{defi:holo}
  For a matrix $A=(a_{ij})\in \mathbb{Z}^{d\times m}$ and
  a vector $\beta\in\mathbb{C}^d$, the
  $A$-{\it hypergeometric system} is the following system of
  linear partial differential equations for an indeterminate
  function $f=f(x)$:
  \begin{align}
    &(\sum_{j\in[m]}a_{ij}x_j\partial_j-\beta_i)f=0,
    \quad  i\in[d], \quad
    \partial_j:=\frac{\partial}{\partial x_j}, \quad
    \text{and}
    \label{HA1}\\
    &(\partial^{z^+}-\partial^{z^-}) f=0,\quad
    z\in{\rm Ker} A\cap {\mathbb{Z}}^m, \quad
    \partial^z:=\prod_{j\in[m]}\partial_j^{z_j},
    \label{HA2}
  \end{align}
  where $z^+_j=\max\{z_j,0\}$ and $z^-_j=-\min\{z_j,0\}$.
  The second group of operators in \eqref{HA2} generates
  the {\it toric ideal} of $A$.
\end{definition}

We know that the $A$-hypergeometric polynomial with
Definition~\ref{defi:toric} satisfies the homogeneity:
\begin{equation}
  (\sum_{j\in[m]}x_j\partial_j-\deg(\beta))Z_A(\beta;x)=0,
  \label{homo}
\end{equation}
and a simple computation shows the $A$-hypergeometric
polynomial satisfies the {\it contiguity relation}:
\begin{equation}
  \partial_jZ_{A}(\beta;x)=Z_{A}(\beta-a_j;x), \quad
  j\in[m].
  \label{cont}
\end{equation}

These properties give the following lemma.
In what follows, we sometimes denote the vector of
the UMVUE $\tilde{\mu}$ by $\tilde{\mu}(b;x)$ to explicitly
show the dependence on sufficient statistics $b$ and $x$.
We apply this rule to the expected counts $\mu$ and
the MLE $\hat{\mu}$ discussed below.

\begin{lemma}\label{lemm:UMVUE}
  The UMVUE $\tilde{\mu}$ of the vector of expected
  counts $\mu$ in the log-affine model $\mathcal{M}_A$ is
  given as
  \begin{equation}\label{UMVUE}
  \tilde{\mu}_j(b;x)=\frac{Z_A(b-a_j;x)}{Z_A(b;x)}x_j,
  \quad j\in[m],
  \end{equation}
  for $b\in\mathbb{N}A$, where $Z_A(b;x)$ is
  the $A$-hypergeometric polynomial defined by \eqref{Apol}.
  Moreover, it is the unique unbiased estimator that is
  a function of $b$.
\end{lemma}  

\begin{proof}
  By the conditional distribution \eqref{cond}, we observe
  that
  \[
  \mathbf{E}(U_j|AU=b)=\frac{1}{Z_A(b;x)}
  \sum_{v\in\mathcal{F}_A(b)} v_j\frac{x^v}{v!}
  =\frac{\partial_j Z_A(b;x)}{Z_A(b;x)}x_j=
  \frac{Z_A(b-a_j;x)}{Z_A(b;x)}x_j, \quad j\in[m],
  \]
  where the last equality follows by the contiguity relation
  \eqref{cont}. This is the unbiased estimator of $\mu_j$ because
  $\mathbf{E}[\mathbf{E}(U_j|AU=b)]=\mathbf{E}[U_j]=\mu_j$.
  Since $b$ is the vector of complete sufficient statistics,
  by the Lehmann-Scheff\'e theorem \cite[Theorem 2.1.11]{LC98},
  it is the unique unbiased estimator that is a function of
  $b$.
\end{proof}

By the homogeneity \eqref{homo}, we have
\[
  |u|=\sum_{j\in[m]}\tilde{\mu}_j(b;x)=
  \sum_{j\in[m]}\frac{x_j\partial_jZ_A(b;x)}{Z_A(b;x)}
  ={\deg}(b).
\]
Hence, we conclude that the transition probability
\eqref{kernel} is the ratio of the UMVUE of the expected
counts to the total number of counts. 

\subsection{The MLE}

We introduce the MLE of the vector of expected counts.
The probability law \eqref{prob} leads to the log-likelihood
\[
  l(\xi;b)=\sum_{i\in[d]} b_i\xi_i-\psi(\xi), \quad
  \psi(\xi):=|u|\log Z(e^\xi,x)+{\rm const.},
\]
where $\xi_i=\log t_i$, $i\in[d]$ are the canonical parameters
of the exponential family. The expected counts $\mu_j$, $j\in[m]$
satisfy
\[
  \sum_{j\in[m]} a_{ij}\mu_j
  :=\sum_{j\in[m]} a_{ij}\mathbf{E}[U_j]
  =\mathbf{E}[B_i]=\frac{\partial\psi}{\partial\xi_i},
\quad i\in[d],
\]
where the last equality holds by
$\mathbf{E}[\partial l/\partial\xi_i]=0$.
The MLEs of the canonical parameters $\hat{\xi}_i$, $i\in[d]$
constitute the unique solution of the estimating equations
\[
  \frac{\partial l(\xi;b)}{\partial\xi_i}=0, \quad
  \text{or} \quad b_i=\frac{\partial\psi}{\partial\xi_i},
  \quad i\in[d].
\]
Therefore, we have
\[
  \sum_{j\in[m]}a_{ij}\hat{\mu}_j=
  \left.\frac{\partial\psi}{\partial\xi_i}\right|_{\xi=\hat{\xi}}=
  b_i, \quad i\in[d].
\]
More precisely, we have the following theorem.

\begin{theorem}[{\cite[Theorem 4.8]{Lau94}}]\label{theo:MLE}
  The MLE $\hat{\mu}$ of the vector of expected counts
  $\mu$ in the log-affine model $\mathcal{M}_A$ satisfies
  \begin{equation}\label{est}
    A\hat{\mu}=b
  \end{equation}
  for a sample of the vector of sufficient statistics $b$.
  Moreover, such MLE is unique. This holds under multinomial
  as well as Poisson sampling.
\end{theorem}

Some components of the MLE become zero if sufficient
statistics lie in the boundary of ${\rm Cone}(A)$.
Related to Definition~\ref{defi:lattice}, we have
the following proposition.

\begin{proposition}\label{prop:bound}
  Consider a matrix $A=(a_{ij})\in\mathbb{Z}^{d\times m}$
  satisfying the conditions in Definition~\ref{defi:toric}
  and a vector of sufficient statistics $b\in\mathbb{N}A$.
  If $b-a_j\not\in\mathbb{N}A$ for some $j$, then
  the component of the MLE $\hat{\mu}_j(b;x)$ vanishes. 
\end{proposition}

\begin{proof}
  Since $b\in\mathbb{N}A$ and $b-a_j\not\in\mathbb{N}A$
  for some $j$, if we write $b=\sum_{k\in[m]}v_ka_k$ for
  some $v\in\mathbb{N}^m$ (such $v$ is generally not unique),
  then $v_j=0$. This fact means that $b$ lies in a face
  of ${\rm Cone}(A)$, or
  \[
  b\in\{Av:v\in\mathbb{N}^m,v_j=0\}
  \subset\{Av:v\in\mathbb{R}^m_{\ge 0},v_j=0\}.
  \]
  On the other hand, $\hat{\mu}$ satisfies \eqref{est},
  which implies $\hat{\mu}_j(b;x)=0$.
\end{proof}  

In Section~\ref{subsect:IPS}, we propose to rewrite
Steps~2 and 3 as follows:
\begin{itemize}
\item [] Step~2': Compute the MLE of the expected count
  $\hat{\mu}_j(\beta;x)$ for each $j\in[m]$ of the 
  log-affine model $\mathcal{M}_A$ with the sufficient
  statistics $\beta$;
\item [] Step 3': Pick $j_t=j$ with probability
  $\hat{\mu}_j(\beta;x)/(n-t+1)$, $j\in[m]$,
\end{itemize}
and discuss the implementation.

Before closing this subsection, we mention an observation about
the UMVUE. Note that the following proposition does not mean
that the UMVUE always coincides with the MLE since the UMVUE
generally does not lie in the model manifold $\mathcal{M}_A$.
If it lies in $\mathcal{M}_A$, it should coincide
with the MLE by Theorem~\ref{theo:MLE}.
We will be back to this issue in Section~\ref{subsect:MLEUM}.

\begin{proposition}\label{prop:0}
  The UMVUE $\tilde{\mu}$ of the vector of expected counts
  $\mu$ in the log-affine model $\mathcal{M}_A$ satisfies
  the estimating equation of the MLE \eqref{est}, that is,
  \begin{equation}\label{est2}
    A\tilde{\mu}=b
  \end{equation}
  for a sample of the vector of sufficient statistics $b$.
\end{proposition}

\begin{proof}
  Applying the annihilator \eqref{HA1} to
  the $A$-hypergeometric polynomial $Z_A(b;x)$
  and using the contiguity relation \eqref{cont}, we have
  \[
    \sum_{j\in[m]} a_{ij}x_jZ_A(b-a_j;x)
    =b_i Z_A(b;x), \quad i\in[d],
  \]
  which is equivalent to \eqref{est2} due to 
  the expression of the UMVUE \eqref{UMVUE}.
\end{proof}

\subsection{The UMVUE and the rational MLE}
\label{subsect:MLEUM}

Step~2' of Algorithm~\ref{algo:0} is immediate if we
know a closed form expression of the MLE in advance.
In this respect, we have the following result based
on the fact that the number of critical points of
the likelihood over the complex space is one if and
only if the map from the observed counts to the MLE
admits a specific form called the Horn uniformization
\cite{Huh14}.
  
\begin{theorem}[\cite{DMS21}]\label{theo:RMLE}
  Consider a sample of counts $u\in\mathbb{N}^m$ following
  the log-affine model $\mathcal{M}_A$.
  It has the rational MLE if and only if there is a vector
  $\lambda\in\mathbb{R}^m$ and a matrix
  $H=(h_{ij})\in\mathbb{Z}^{l\times m}$ such that 
  $\mathcal{M}_A$ is the image of the rational map
  $\Phi:\mathbb{R}^m_{>0}\dasharrow\mathbb{R}^m_{>0}$:
  \[
    \Phi_j(u_1,\ldots,u_m)
    =\lambda_j\prod_{i=1}^l(\sum_{k=1}^mh_{ik}u_k)^{h_{ij}},
   \quad j\in[m]
  \]
  with $\sum_{j\in[m]}\Phi_j(u)=1$.
\end{theorem}

Here, we continue Example~\ref{exam:2x2:1}, which
was discussed in \cite{Huh14,DMS21}.

\begin{example}[Two-way contingency tables of the independence model, cont.]\label{exam:2x2:2}
  We have the rational MLE of the expected count
  $\hat{\mu}(b;1)=|u|\Phi(u)$ with
  \[
    \Phi(u)=\left(
    \frac{u_{1\cdot}u_{\cdot 1}}{|u|^2},
    \frac{u_{1\cdot}u_{\cdot 2}}{|u|^2},
    \frac{u_{2\cdot}u_{\cdot 1}}{|u|^2},
    \frac{u_{2\cdot}u_{\cdot 2}}{|u|^2}\right),
  \]
  and from this expression, we can read off 
  \[
    \lambda=(4,4,4,4) \quad \text{and} \quad
    H=\begin{pmatrix}
    1&1&0&0\\0&0&1&1\\-2&-2&-2&-2\\1&0&1&0\\0&1&0&1
    \end{pmatrix},
  \]
  where the rows are for $u_{1\cdot}$,
  $u_{2\cdot}$, $|u|$, $u_{\cdot1}$, and $u_{\cdot2}$,
  respectively. The transition probability of Step~3'
  for the cell at the $j_1$-th row and
    $j_2$-th column (the $\{2(j_1-1)+j_2\}$-th column of
    configuration matrix $A$) becomes
  \begin{equation}\label{rmle_2way}
    \frac{\hat{\mu}_{j_1j_2}(b;1)}{|u|}
    =\frac{u_{j_1\cdot}}{|u|}\frac{u_{\cdot j_2}}{|u|},
    \quad j_1,j_2\in[2].
  \end{equation}
  This condition states that the manifold $\mathcal{M}_A$
  (in the complex space) can be obtained as the Horn uniformization
  of an irreducible polynomial called the $A$-discriminant.
  Here, there exists a matrix $B\in\mathbb{Z}^{r\times l}$
  of rank $r$, such that $BH=0$, which defines the $A$-discriminant.
  In this example,
  \[
  B=\begin{pmatrix}1&1&1&1&1\\0&0&1&2&2\end{pmatrix}.
  \]
  The matrix $B$
  defines the set of polynomials in $t$: $F(t)=(x_1+x_2)+x_3t+(x_4+x_5)t^2$.
  The $A$-discriminant is an irreducible polynomial that
  defines the set of polynomials with multiple roots,
  $x_3^2-4(x_1+x_2)(x_4+x_5)$. See \cite{Huh14} for details.
  In the literature on the $A$-discriminant, we write this $B$
  as $A$, but here we avoid a notational conflict. Note that
  the $A$ in the $A$-discriminant (here, it is denoted by $B$),
  and the $A$ in the $A$-hypergeometric system (Definition~\ref{defi:holo}) are different. 
\end{example}  

In the remainder of this subsection we concentrate
on the log-linear model, so we set $x=1$ in
Definition~\ref{defi:toric}.

Geiger et al.~\cite{GMS06} showed that
an undirected graphical model for discrete variables
has a rational MLE if and only if it is
a decomposable graphical model 
(an undirected graphical model for an undirected graph
$G$ is the log-linear model in which the sufficient
statistics correspond to the cliques of $G$.
see \cite{Lau94}
for the definition of the decomposable graphical model).
Using this fact, we have the following theorem.

The notation for graphical models can be complex;
we follow that of \cite{Lau94} here. The correspondence
between this notation and the notation of
the $A$-hypergeometric polynomial is explained in
detail in \cite[Section 3.2]{MT21+}, so we will
not repeat it here.

\begin{theorem}\label{theo:MLEUM}
  The UMVUE and the MLE of an undirected discrete graphical model
  coincide if and only if the model is a decomposable graphical
  model. Moreover, the MLE is rational.
\end{theorem}

\begin{proof}
  (if) Lemma~4.21 of \cite{Lau94} gives the normalization
  constant of a decomposable graphical model. We can
  represent it as the $A$-hypergeometric polynomial,
  and we have
  \[
  Z_A(b;1)=
  \frac{\prod_{S\in\mathcal{S}}\{\prod_{j_S}u(j_S)!\}^{\nu(S)}}
  {\prod_{C\in\mathcal{C}}\prod_{j_C}u(j_C)!},
  \]
  where $\mathcal{S}$ and $\mathcal{C}$ are the sets
  of separators and cliques of the undirected graph $G$,
  respectively, and $\nu(S)$ is the number of appearances
  of separator $S$ in a perfect sequence of the cliques,
  and $b$ is the vector consisting of
  the sequence of $u(j_C)$, $C\in\mathcal{C}$.
  The index $j$ specifies a state of the vertex set of $G$.
  The indices $j_S$ and $j_C$ are indices obtained by summing
  up indices of $j$ other than $S$ and $C$, respectively,
  and $u(j_S)$ and $u(j_C)$ are the number of counts of
  states consistent with $j_S$ and $j_C$, respectively.
  Lemma~\ref{lemm:UMVUE} shows that the UMVUE is 
  \begin{align}
    \tilde{\mu}_j(b;1)&=\frac{Z_A(b-a_j;1)}{Z_A(b;1)}
    \nonumber\\
  &=
  \frac{\prod_{S\in\mathcal{S}}
    \prod_{i_S\neq j_S}\{u(i_S)!\}^{\nu(S)}
    \{(u(j_S)-1)!\}^{\nu(S)}}
  {\prod_{C\in\mathcal{C}}
    \prod_{i_C\neq j_C}u(i_C)!(u(j_C)-1)!}
   \frac{\prod_{C\in\mathcal{C}}\prod_{i_C}u(i_C)!}
    {\prod_{S\in\mathcal{S}}
      \prod_{i_S}\{u(i_S)!\}^{\nu(S)}}\nonumber\\
  &=\frac{\prod_{C\in\mathcal{C}}u(j_C)}
    {\prod_{S\in\mathcal{S}}u(j_S)^{\nu(S)}}.
  \label{umvue_dec}  
  \end{align}
  The right-hand side is the expression of the MLE
  $\hat{\mu}_j(b;1)$ given in Proposition~4.18
  of \cite{Lau94} and the rational
  function of observed counts.\\
  (only if) Suppose we have a non-decomposable graphical
  model. Then, Theorem~4.4 of \cite{GMS06} implies that
  we cannot always find a rational solution to the
  maximum likelihood estimation. On the other hand,
  the expression of the UMVUE in Lemma~\ref{lemm:UMVUE}
  implies that the UMVUE is always a rational number,
  because of the definition of the $A$-hypergeometric
  polynomial \eqref{Apol}. Therefore, the MLE and
  the UMVUE never coincide.
\end{proof}  

With Proposition~\ref{prop:0}, we immediately have the following
corollary.

\begin{corollary}\label{coro:MLEUM}
  The UMVUE of an undirected discrete graphical model
  lies in the model manifold $\mathcal{M}_A$ if and only
  if the model is a decomposable graphical model.
\end{corollary}

\begin{example}[Two-way contingency tables of the independence model, cont.]\label{exam:2x2:3}
  The two-way contingency table of the independence
  model is a decomposable graphical model:
  the undirected graph consists of two vertices, say 1 and 2,
  corresponding to the row and column, respectively.
  It has no edges. A perfect sequence of the cliques is
  $C_1=\{1\}$, $C_2=\{2\}$ with separator $S=\emptyset$,
  $\nu(S)=1$. If we write $j=j_1j_2$,
  $j_{C_1}=j_1\cdot$ and $j_{C_2}=\cdot j_2$, we have
  $u(j_{C_1})=u_{j_1\cdot}$, $u(j_{C_2})=u_{\cdot j_2}$, and
  $u(S)=|u|$ (all indices are summed up).
  $b=(u_{1\cdot},u_{2\cdot},u_{\cdot1},u_{\cdot2})^\top$.
  Then, \eqref{umvue_dec} gives the UMVUE
  \[
  \tilde{\mu}_{j_1j_2}(b;1)=\frac{u(j_{C_1})u(j_{C_2})}{u(S)}
  =\frac{u_{j_1\cdot}u_{\cdot j_2}}{|u|},
  \]
  which coincides with the rational MLE appeared in
  \eqref{rmle_2way}.
\end{example}  

Note that Theorem~\ref{theo:MLEUM} is a statement limited to
undirected discrete graphical models, and does not provide
a necessary and sufficient condition for the UMVUE and the MLE
to coincide in log-linear models. In fact, as the following
example shows, the UMVUE and the MLE can coincide in log-linear
models that are not graphical models.

\begin{example}[Two-way contingency tables of the quasi-independence models.]\label{exam:q2x2}
  Coons and Sullivant~\cite{CS21} classified the two-way
  contingency tables of the quasi-independence models
  (independence models with structural zeros) that have
  rational MLEs. The quasi-independence models are not
  graphical models, since the correspondence between
  the sufficient statistics and the cliques is incomplete.
  In Example~2.5 of \cite{CS21}, the rational MLE
  of the $3\times 3$ table whose cell at the third row and
  the third column is zero was given. The table is
  \[
    \begin{array}{ccc|c}
    u_{11}&u_{12}&u_{13}&u_{1\cdot}\\
    u_{21}&u_{22}&u_{23}&u_{2\cdot}\\
    u_{31}&u_{32}& 0    &u_{31}+u_{32}\\    
    \hline
    u_{\cdot1}&u_{\cdot2}&u_{13}+u_{23}&|u|
    \end{array}.
  \]
  The configuration matrix, the vectors of the observed counts,
  and the vector of the sufficient statistics are
  \begin{equation*}
    A\!=\!\left(\begin{array}{cccccccc}
      1&1&1&0&0&0&0&0\\
      0&0&0&1&1&1&0&0\\
      0&0&0&0&0&0&1&1\\
      1&0&0&1&0&0&1&0\\
      0&1&0&0&1&0&0&1\\
      0&0&1&0&0&1&0&0\\
    \end{array}\right),\quad
    u\!=\!\left(
    \begin{array}{c}
      u_{11}\\u_{12}\\u_{13}\\u_{21}\\u_{22}\\u_{23}\\
      u_{31}\\u_{32}
    \end{array}
    \right),\quad {\rm and}\quad
    b=\left(
    \begin{array}{c}
      u_{1\cdot}\\u_{2\cdot}\\u_{31}+u_{32}\\u_{\cdot1}\\u_{\cdot2}\\u_{13}+u_{23}
    \end{array}
    \right),
  \end{equation*}
  respectively. The rational MLE of the expected count of
  the first row  and the third column is
  \[
    \hat{\mu}_{13}(b;1)=\frac{u_{1\cdot}(u_{13}+u_{23})}{u_{1\cdot}+u_{2\cdot}},
  \]    
  which coincides with the UMVUE. See Appendix for the details.
\end{example}    

\subsection{The approximate algorithm without using
  connection matrices}
\label{subsect:IPS}

In this subsection, we propose to use Algorithm~\ref{algo:0}
with Steps 2 and 3 replaced by Steps 2' and 3', respectively.
We call the resulting algorithm
the {\it approximate direct sampling algorithm}. From
Lemma~2.4, when the MLE is not the UMVUE, i.e., it is not
the rational MLE, then there exists a bias that disappears
asymptotically. Thus the resulting algorithm is approximate
in the sense that it is not exact.

Before proceeding, we summarize what we know about
Algorithm~\ref{algo:0} according to Theorem~\ref{theo:MLEUM}
as a proposition.

\begin{proposition}\label{prop:exact}
  For an undirected discrete graphical model,
  the approximate algorithm obtained
  from Algorithm~\ref{algo:0} by replacing Steps 2 and 3 with
  Steps 2' and 3' coincides with the original
  Algorithm~\ref{algo:0} if and only if the model
  is a decomposable graphical model. Then, the algorithm is
  exact. 
\end{proposition}

Theorem~\ref{theo:RMLE} means that the time computational
complexity of Step~2 of Algorithm~\ref{algo:0} is $O(lm)$
if we use the rational MLE.
  
\begin{remark}
  The fact that Step~2 of Algorithm~\ref{algo:0} is immediate
  by the rational MLE for decomposable graphical models is
  not surprising, given that we can sample directly by combining
  urn schemes (see, e.g., \cite[Section 9.6]{Sul18}). However,
  an implementation of Algorithm~\ref{algo:0} may be easier.
\end{remark}

Combining Remark~\ref{rema:hit} and Proposition~\ref{prop:bound},
we arrive at the following consequence: if sufficient
statistics lie in the boundary of the polyhedral cone generated
by the column vectors of $A$, the MLE coincides with the UMVUE and
vanishes exactly without bias. In practical terms, we obtain
the following proposition.

\begin{proposition}\label{prop:lattice}
  The approximate algorithm obtained from Algorithm~\ref{algo:0}
  by replacing Steps 2 and 3 with Steps 2' and 3' is a Markov
  chain on the Markov lattice defined in Definition~\ref{defi:lattice}.
\end{proposition}

\begin{proof}
  (Markov property) Even in the algorithm where Steps~2 and 3
  are replaced by Steps~2' and 3', the transition probability
  remains determined by the current state, so the Markov
  property holds.\\
  (agreement of positivity of transition probabilities)
  Proposition~\ref{prop:bound} asserts that when a component
  of the UMVUE vanishes, that component of the MLE also vanishes,
  and vice versa. Therefore, for a transition in which
  the transition probability of Algorithm~\ref{algo:0} is 0,
  the transition probability of the approximate algorithm is
  also 0, and vice versa. Furthermore, for a transition in
  which the transition probability of Algorithm~\ref{algo:0}
  is positive, the transition probability of the approximate
  is also positive, and vice versa.\\
  (agreement of supports of sample paths)
  Let $(j_1,...,j_n)$
  denote a sample path obtained with positive probability in
  Algorithm~\ref{algo:0}, and $p_{j_t,j_{t+1}}$, $t=1,2,...,n-1$
  denote the transition probabilities of each step.
  Since the algorithm produces a Markov chain, the probability
  of obtaining the sample path is given by the product of these,
  so $p_{j_t,j_{t+1}}$ must all be positive. The transition
  probability of each step of the approximate algorithm
  is expressed as $p'_{j_t,j_{t+1}}$, $t=1,2,...,n-1$.
  In this case, from the agreement of positivity of transition
  probabilities shown above, we can see that the probability of
  obtaining the sample path $(j_1,...,j_n)$, namely, the product
  of $p'_{j_t,j_{t+1}}$, $t=1,2,...,n-1$, is positive. Conversely,
  a sample path in the approximate algorithm is also obtained
  with a positive probability in
  Algorithm~\ref{algo:0}. Thus, we conclude that whether
  a certain sample path is positive or zero agrees between
  Algorithm~\ref{algo:0} and the approximate algorithm.
  Since the Markov lattice is the support of sample paths of
  Algorithm~\ref{algo:0}, it is also the support of sample
  paths of the approximate algorithm.
\end{proof}

\begin{remark}\label{rema:bias}
In general, the MLE has asymptotic normality; that is,
$\sqrt{n}(\hat{\mu}-\mu)$ converges in distribution
to the normal distribution of expectation zero, where $n$
is the total number of counts. As a result, the MLE is
asymptotically unbiased. For the covariance of the
normal distribution, see \cite{Lau94}, page 78.
However, this does not mean that the approximate
algorithms is unbiased asymptotically when $n$ is large.
This is because the algorithm proceeds by reducing
the total number of counts in the contingency table
with each transition (see the tables in the first line
of \eqref{tables}), so the total number of counts in
the contingency tables at later steps of the algorithm
becomes small. See Section~\ref{sect:exp} for numerical examples.
\end{remark}

For any log-affine model, we can always evaluate the MLE
numerically solving the estimating equation \eqref{est}
by using the {\it iterative proportional scaling}
(IPS), specifically called the generalized IPS by Darroch
and Ratcliff \cite{DR72}. We rewrite the procedure in our
notation. Consider a sample of counts $u\in\mathbb{N}^m$
following the log-affine model $\mathcal{M}_A$
satisfying $Au=\beta$, where $\beta\in\mathbb{N}A$.
For simplicity, we assume $A\in\mathbb{N}^{d\times m}$,
$\beta_i>0, i\in[d]$, and all columns of $A$ have the same
sum denoted by $s$, but a general form is available in
\cite{DR72}. That is
\begin{align}
  &p_j^{(0)}=\frac{x_j}{\sum_{k\in[m]}x_k}, \quad j\in[m], \quad \text{and} \nonumber\\
  &p_j^{(t+1)}=p_j^{(t)}
  \left\{
  \prod_{i\in[d]}
  \left(
  \frac{\beta_i/|u|}{\sum_{k\in[m]} a_{ik}p_k^{(t)}}\right)^{a_{ij}}
  \right\}^{1/s}, \quad j\in[m].\label{ips2}
\end{align}  
Theorem~1 of \cite{DR72} guarantees the convergence:
$\lim_{t\to\infty}|u|p^{(t)}=\hat{\mu}(\beta;x)$.

The time computational complexity of Step~2' of
Algorithm~\ref{algo:0} is $O(dm\tau)$ if we use the generalized
IPS, where $\tau$ is the minimum number of iterations to
achieve convergence with predetermined precision.
In the numerical experiments reported in Section~\ref{sect:exp}, we
used the iteration \eqref{ips2} with a convergence criterion to stop
the iteration at the time
\begin{equation}\label{conv}
  \tau:=\min\left\{t\ge 1:\sum_{i=1}^d||u|\sum_{k\in[m]}a_{ik}p_k^{(t)}-\beta_i|<\epsilon d\right\}
\end{equation}
for some small $\epsilon>0$, which means the error per a 
sufficient statistic is less than $\epsilon$.

\begin{remark}
  There are various types of IPS algorithms for maximum likelihood
  estimation, and we may use any of them.
  The classical IPS known as the Deming--Stephan
  algorithm for log-linear models has the iteration with setting
  $s=1$ in \eqref{ips2}. Coons et al.~\cite{CLR24} considered
  a subfamily of the log-linear model and obtained a sufficient
  condition under which the classical IPS produces the rational
  MLE in one cycle of iteration. 
\end{remark}  

We summarize the computational cost of the 
algorithms proposed in this paper. We only consider the time
complexity because the space complexity, i.e., the memory
requirement of the implementation, is small enough that they
are of no concern.

\begin{proposition}
  The time computational complexity of Algorithm~\ref{algo:0}
  is $O(lm|u|)$ if we use the rational MLE for Step~2, where
  $l$ is in Theorem~\ref{theo:RMLE}. For the approximate
  algorithm obtained from Algorithm~\ref{algo:0} by replacing
  Steps~2 and 3 with Steps~2' and 3', the complexity is
  $O(dm\tau|u|)$ if we use the generalized IPS for Step~2'.
  Here, \eqref{conv} defines $\tau$.
\end{proposition}

In the approximate algorithm, we avoided the need to compute
the connection matrices of the holonomic ideal generated by
the $A$-hypergeometric system. Therefore, we can avoid the
computation of Gr\"obner bases or investigation of specific
properties of each model for any model. We explain how we
achieved this before closing this section.

For the holonomic ideal, we have a {\it Pfaffian system} of equations:
\begin{equation}\label{pfaff}
  x_j\partial_jq(\beta;x)=F_j(\beta;x)q(\beta;x), \quad
  j\in[m]
\end{equation}
for some connection matrices $F_j(\beta;x)$, $j\in[m]$ and the vector
\[
  q(\beta;x)
  =(Z_A(\beta;x),\partial^{u(1)}Z_A(\beta;x),\ldots,\partial^{u(r-1)}Z_A(\beta;x))^\top,
\]
where $\{1,\partial^{u(1)},\ldots,\partial^{u(r-1)}\}$
is the set of the standard monomials of the holonomic ideal
of rank $r$, where $\partial^{u(\cdot)}$ is a monomial in
$\partial_j$, $j\in[m]$. The standard monomials provide
the basis of the solutions of the $A$-hypergeometric
system. The system of equations \eqref{pfaff} can be obtained
via a normal form with respect to a Gr\"obner basis of
the holonomic ideal (see, e.g., \cite[Theorem 1.4.22]{SST00}).
Then, with the contiguity relation \eqref{cont}, we obtain
a recurrence relation for the vectors $q(\beta;x)$:
\begin{equation}\label{rec2}
  x_jq(\beta-a_j;x)=F_j(\beta;x)q(\beta;x), \quad
  j\in[m],
\end{equation}
which was originally used to compute the $A$-hypergeometric
polynomial $Z_A(\beta;x)$ \cite{OT15} and used in a previous
implementation of the direct sampling algorithm
(\cite{MT21+}, Algorithm~2.2). Nevertheless,
Lemma~\ref{lemm:UMVUE} is equivalent to a relation
\[
x_jZ_A(\beta-a_j;x)=\tilde{\mu}_j(\beta;x)Z_A(\beta;x), \quad
j\in[m].
\]
Comparing this with \eqref{rec2}, we see that
by appropriate basis transformation we can make
the first row of the connection matrix to be
$(\tilde{\mu}_j(\beta;x),0,\ldots,0)$. All we need is
an evaluation of the UMVUE $\tilde{\mu}$ exactly or
approximately, since the recurrence relation for
the first element of the transformed vector decouples
from the others. Here, the MLE asymptotically coincides
with the UMVUE and can be used as an approximation of
the UMVUE, but we can compute the MLE without using
the recurrence relation.

\section{Numerical experiments}\label{sect:exp}

We applied the direct sampling algorithm and the approximate
algorithm introduced in Section~\ref{subsect:IPS} to three
log-affine models. The first two models are two-way contingency
tables. Since the independence model introduced in
Example~\ref{exam:2x2:1} is a decomposable graphical model,
the UMVUE coincides with the rational MLE
by Theorem~\ref{theo:MLEUM}, and the original
Algorithm~\ref{algo:0} works efficiently.
The other is a non-independence model, which is not log-linear.
In Section~4.2 of \cite{MT21+}, the authors applied a previous
implementation of the direct sampling algorithm
(Algorithm 2.2 of \cite{MT21+}) to the non-independence model.
It is exact, but they found it was inefficient due to
the computational load. The last model is a no-three-way
interaction model for the three-way contingency tables. It is
log-linear, but not a graphical model (the edges corresponding
to the two-way interactions make the complete graph of
the three vertices, but the model has no three-way interaction),
and one can see that the UMVUE and the MLE do not coincide.
In Section~4.3 of \cite{MT21+}, the authors discussed
the model but did not experiment since it
involves the recurrence relation \eqref{rec2} for a vector of
81 dimensions, and its computation was prohibitive.

We implemented the direct sampling and the Metropolis
algorithms on a computer and compared their performance
because they are equally easy to implement, and
the Metropolis algorithm is the standard in practitioners.

To evaluate the performance of the Metropolis algorithm,
it is first necessary to assess the convergence of
the Metropolis chain to the stationary distribution.
It is a nontrivial task with various arguments
(see, e.g., \cite[Section 12.2]{RC04}). In this study,
taking advantage of the situation where we have another
feasible sampling algorithm, we assessed the convergence of
the Metropolis chain using the direct sampling algorithm.

Since the number of possible contingency tables is enormous,
we employed the distribution of the chi-square values of
the contingency table rather than directly assessing
the distribution of contingency tables. We estimated
the distribution of chi-square values under the conditional
distribution \eqref{cond} of contingency tables, that is,
\begin{equation}\label{chi2}
  p(z):=\mathbf{P}(\chi^2(U)=z|AU=b), \quad z\ge 0,
\end{equation}
by the empirical distribution $\hat{p}_{\rm M}$ of contingency
tables taken from the Metropolis chain and the empirical
distribution $\hat{p}_{\rm D}$ of contingency tables taken by
the direct sampling algorithm. We may regard these estimates as
the Monte-Carlo integrations of the indicator function
$1\{\chi^2(\cdot)=z\}$. Then, we computed the total variation
distance between them:
\begin{equation}\label{TV}
  \|\hat{p}_{\rm M}-\hat{p}_{\rm D}\|_{\rm TV}
  =\frac{1}{2}\sum_z|\hat{p}_{\rm M}(z)-\hat{p}_{\rm D}(z)|^2.
\end{equation}
Here, we generated many contingency tables for $\hat{p}_{\rm D}$
in advance so that the empirical distribution of the direct
sampling algorithm gives the best possible approximation of
the truth \eqref{chi2}. If we use the (exact) direct sampling
algorithm, \eqref{TV} approximates the total variation
distance between the distribution of chi-square values taken
from the Metropolis chain and that of the stationary distribution.
Otherwise, we can still assess the convergence of the Metropolis
chain because the total variation distance should converge to
a non-zero value.

Contingency tables taken from the Metropolis chain are
not independent. To account for the autocorrelation among
contingency tables, we adopted the concept of the
{\it effective sample size} (ESS). The ESS of $N$ consecutive
steps in a Metropolis chain is
$N/(1+2\sum_{t\ge 1}\rho_t)$, where $\rho_t$ is
the autocorrelation of chi-square values at lag $t$. We used
sample autocorrelation and censored the sum where the sample
autocorrelation was less than 0.01.

The results of the numerical experiments described in this
section reveal the following:

\begin{itemize}
  
\item [(1)] For the non-independence model of two-way
  contingency tables, which is a decomposable graphical model,
  the direct sampling algorithm runs faster than the Metropolis
  algorithm. 

\item[(2)] For general log-affine models, the direct sampling
  algorithm requires the calculation of transition probabilities.
  It is different from the Metropolis algorithm, where we can
  use the same Markov basis as the corresponding log-linear
  model ($x=1$ in Definition~\ref{defi:toric}).
  For a non-independence model of two-way contingency tables,
  the approximate direct sampling algorithm is several tens of
  times faster than the exact direct sampling algorithm.
  Although it needs a larger amount of time than the Metropolis
  algorithm in terms of the effective sample size, it can be
  calculated in a similar amount of time. If the number of
  contingency tables is fixed, as is often done when using
  the Metropolis algorithm, the approximate algorithm achieves
  sampling closer to the true distribution than the Metropolis
  algorithm.

\item[(3)] As an example of a non-graphical
  log-linear model, we examined the no-three-way interaction
  model of three-way contingency tables. Its Markov basis is known.
  The exact direct sampling algorithm was infeasible because
  of the computational burden. The approximate direct sampling
  algorithm required several tens of times longer than
  the Metropolis algorithm in terms of the effective sample
  size. However, note that the approximate direct sampling
  algorithm works whenever the MLE is available.
  The Metropolis algorithm cannot be applied unless we know
  a Markov basis.

\end{itemize}

We conducted all computations in the {\tt R} statistical
computing environment, using floating-point arithmetic,
except for the results shown at the end of
Section~\ref{subsect:2way}.
The following timing results were taken on a single
core of the CPU (Intel Core i5-4308U CPU, 2.8GHz) of
a machine with 8GB memory.

\subsection{Two-way contingency tables}\label{subsect:2way}

We investigated the performance of the direct sampling
algorithm by exemplifying the two-way contingency tables.
The independence model is a log-linear decomposable model
(see Example~\ref{exam:2x2:3}), while the non-independence
model is a log-affine model but not a log-linear model.
We investigated following two-way contingency tables with
fixed marginal sums for $s=1,2,5$ and $10$:
\[
\begin{array}{ccccc|c}
    u_{11}&u_{12}&u_{13}&u_{14}&u_{15}&5s\\
    u_{21}&u_{22}&u_{23}&u_{24}&u_{25}&5s\\
    u_{31}&u_{32}&u_{33}&u_{34}&u_{35}&5s\\
    u_{41}&u_{42}&u_{43}&u_{44}&u_{45}&5s\\
    \hline
    4s&4s&4s&4s&4s&20s
  \end{array},
\]
where the vector
\[
  u=(u_{11},\ldots,u_{15},u_{21},\ldots,u_{25},\ldots,
  u_{41},\ldots,u_{45})^\top,
\]
is random, the vector of sufficient statistics is
\[
  b=(u_{1\cdot},\ldots,u_{4\cdot},u_{\cdot1},\ldots,u_{\cdot5})^\top
  =(5s,\ldots,5s,4s,\ldots,4s)^\top,
\]
and the configuration matrix is
\[
  A=\left(\begin{array}{c}
  E_4 \otimes 1_5\\
  1_4 \otimes E_5
  \end{array}\right),  \quad
  (e_r)_{ij}:=\delta_{i,j},~i,j\in[r],
  \quad 1_r:=(\overbrace{1,\ldots,1}^r).
\]
These tables are unrealistically symmetric, but advantages of this
symmetry for the following investigations are that we know
the expected counts of each cell for the independent model is
uniform, and we expect that the Metropolis chain starting from
this tables will mix rapidly. The rational MLE of the expected
count is \eqref{rmle_2way} with $i\in[4]$ and $j\in[5]$.
For the non-independence model, we set the odds ratios
($x$ in Definition~\ref{defi:toric}) as
\begin{equation}\label{odds}
  \left(
  \begin{array}{ccccc}
    3&2&1&1&1\\
    2&2&1&1&1\\
    1&1&1&1&1\\
    1&1&1&1&1\\
  \end{array}
  \right).
\end{equation}
For the generalized iterative proportional scaling (IPS), we adopt
the convergence criterion \eqref{conv} with $\epsilon=0.1$.
All computations involved converged within 20 steps.

First, we generated one million contingency tables using the direct
sampling algorithm (exact and approximate for the independence and
non-independence models, respectively). We computed the empirical
distribution $\hat{p}_{\rm D}$ of the chi-square values, which should
be a good approximation to the stationary distribution for
the independence model.

Then, we assessed the convergence of the Metropolis chain
to the stationary distribution. For the Metropolis chain of four
pairs of the burn-in and length:
$(0,10^3)$, $(10^3,10^4)$, $(10^4,10^4)$, and $(10^4,10^5)$,
respectively, we computed the empirical distribution
$\hat{p}_{\rm M}$ of the chi-square values and calculated
the total variation distance \eqref{TV}. Here, we set the expected
value of each cell in the chi-square to $s$ for both
the independence and non-independence models with the odds
ratios \eqref{odds}. We show the averages of the results obtained
by repeating this procedure 10 times in Table~\ref{tabl:1}. The pair
$(0,10^3)$ gave the poorest results because only a few
contingency tables appeared in a Metropolis chain, and they
cannot approximate the stationary distribution. The pairs
$(10^3,10^4)$ and $(10^4,10^4)$ were equally good, which
implies that the Metropolis chain converged after 10,000
steps. For the pair $(10^4,10^5)$, the approximation is even
better because the empirical distribution approximates
the stationary distribution better by many contingency tables
because of the Glivenko--Cantelli theorem. These trends were
similar for non-independence models.
Another important observation is that for approximate sampling
of the non-independence models, the total variation distance
did not decrease as $s$ increases.
Although one might expect the total variation distance
to decrease as the bias of the MLE decreases with increasing
sample size (see Remark~\ref{rema:bias}), the observed
result is the opposite: the total variation distance
increases. This contrast indicates that the Monte Carlo
error dominates the bias, making the impact of the bias
comparatively minor. Indeed, the total variation distances showed
a similar trend for the direct sampling from the independence
model, where the direct sampling is exact without bias.
\begin{table}[t]
  \small
  \caption{The total variation distances between the
    empirical distributions of chi-square values generated
    by the Metropolis algorithm and that by the direct sampling
    algorithm (exact and approximate for the independence
    and non-independence models, respectively).}
  \label{tabl:1}\smallskip
  \centering
  \begin{tabular}{cccccc}
  &     &\multicolumn{4}{c}{(burn-in, length)}\\
  Model &$s$&$(0,10^3)$&$(10^3,10^4)$&$(10^4,10^4)$&$(10^4,10^5)$\\
  \hline  
  Independence
        &1  &0.133 &0.039 &0.035 &0.012\\ 
        &2  &0.137 &0.041 &0.043 &0.015\\
        &5  &0.192 &0.055 &0.062 &0.019\\
        &10 &0.236 &0.072 &0.073 &0.025\\
  \hline  
  Non-independence
        &1  &0.107 &0.036 & 0.031 &0.017\\
              &2  &0.141 &0.047 & 0.046 &0.025\\
              &5  &0.190 &0.063 & 0.065 &0.029\\
              &10 &0.273 &0.089 & 0.083 &0.037\\
    \hline
  \end{tabular}  
\end{table}

For the Metropolis algorithm, the same Markov basis can
be applied whether the model is independent or not, and
the computational cost does not change significantly.
On the other hand, for the direct sampling algorithm,
the cost of calculating the transition probabilities is
an issue. For the independence model, the transition
probabilities are given by the rational MLE and can be
calculated immediately, but for the non-independence model,
calculating the transition probabilities takes time.
The purpose here is to evaluate the computational cost.

Given the above results, we drew 10,000 tables after
the burn-in of 10,000 steps for the Metropolis chain.
We computed the ESS of the 10,000 tables drawn by
the Metropolis algorithm and drew the same number of
tables using the direct sampling algorithm.
Table~\ref{tabl:2} shows the averages of the results
obtained by repeating this procedure 100 times. It shows
the times it took to draw 10,000 tables using the Metropolis
algorithm (including the burn-in), the time it took to draw
the tables of the number of the ESS using the direct sampling
algorithm, and one table using the direct
sampling algorithm. For the independence model, the direct
sampling algorithm was uniformly more efficient
than the Metropolis algorithm. For the non-independence model,
although the direct sampling algorithm was less
  efficient than the Metropolis
algorithm, even when the total count is 200 ($s=10$), which
is the most computationally expensive for the direct sampling
algorithm, the direct sampling algorithm took only about six
times as long as the Metropolis algorithm.
\begin{table}[t]
  \small
  \caption{The times to draw 10,000 tables by the Metropolis algorithm,
    the effective sample sizes (ESS) of the 10,000 tables,
    the times to draw the same number of tables by the direct
    sampling algorithm (exact and approximate for the independence
    and non-independence models, respectively), and the times per single
    table. Times are in seconds.}
  \label{tabl:2}\smallskip
  \centering
  \begin{tabular}{cccccc}
Model  &$s$ &ESS &Metropolis & Direct & 1 table\\
\hline
Independence &1   &327 &0.604 &0.193&0.0006\\ 
(Exact) &2   &265 &0.663 &0.308&0.0012\\
             &5   &155 &0.750 &0.448&0.0029\\
             &10  &103 &0.839 &0.594&0.0058\\
\hline
Non-independence &1  & 327 & 0.650 & 1.543 & 0.005\\
(Approximate) &2  & 257 & 0.726 & 2.876 & 0.011\\
             &5  & 145 & 0.846 & 4.683 & 0.032\\
             &10 & 81  & 0.945 & 5.711 & 0.071\\
\hline
\end{tabular}
\end{table}

At the end of this subsection, we compare the accuracy and
computational efficiency of the (exact) direct sampling algorithm
and the approximate algorithm for the non-independent model.
We performed calculations using rational arithmetic in
the {\tt Risa/Asir} computational algebra system \cite{NT92},
since the direct sampling algorithm is implemented on
{\tt Risa/Asir}. Since Algorithm~1 was infeasible for
the contingency tables discussed above, we investigated
the following two-way contingency tables with fixed marginal
sums for $s=1,2,5,10$:
\[
\begin{array}{cccc|c}
    u_{11}&u_{12}&u_{13}&u_{14}&4s\\
    u_{21}&u_{22}&u_{23}&u_{24}&4s\\
    u_{31}&u_{32}&u_{33}&u_{34}&4s\\
    \hline
    3s&3s&3s&3s&12s
  \end{array}
\]
with setting the odds ratios as
\[
  \left(
  \begin{array}{cccc}
    3&2&100/101&1\\
    200/104&200/103&100/102&1\\
    1&1&1&1\\
  \end{array}
  \right).
\]
\begin{table}[t]
  \small
  \caption{The times to draw 1,000 and
    10,000 tables by the exact direct sampling
    algorithm and the approximate direct sampling
    algorithm, respectively, for the non-independence
    model. The total variation distances between
    the empirical distributions of chi-square values
    generated by these algorithms are also shown.
    The total variation distances between the exact
    direct sampling algorithm and the Metropolis
    algorithm are also shown. Times are in seconds.}
  \label{tabl:3}\smallskip
  \centering
  \begin{tabular}{cccccc}
    & Exact & \multicolumn{2}{c}{Approximate} & Metropolis\\
$s$ & Time  & Time & Distance & Distance \\
\hline
1   & 1581  &  308  & 0.019 & 0.032 \\
2   & 3496  &  721  & 0.051 & 0.054 \\
5   & 8513  & 2171  & 0.096 & 0.099 \\
10  & 18270 & 5170  & 0.138 & 0.151 \\
\hline
\end{tabular}
\end{table}

We generated 1,000 and 10,000 tables using the exact
direct sampling algorithm and the approximate direct
sampling algorithm, respectively. Table~\ref{tabl:3}
shows the times they took to draw tables, and
the results for the approximate direct sampling
algorithm are the averages of the results obtained
by repeating this procedure 10 times. The approximate
direct sampling algorithm performed around 35-50
times faster than the exact direct sampling algorithm.
The total variation distances between the empirical
distributions of chi-square values generated by
the exact and approximate direct sampling algorithms,
as well as between the exact direct sampling algorithm
and the Metropolis algorithm, are also shown.

For the Metropolis chain, 10,000 tables were drawn
following a 10,000-step burn-in. The shown values are
averages across 10 repetitions. Although the Metropolis
algorithm approximates the stationary distribution
well, as shown by the results in Table~\ref{tabl:2}
for larger contingency tables, the approximate direct
sampling algorithm gave a distribution closer to that
of the exact direct sampling algorithm than
the Metropolis algorithm did. It is notable that
the approximate algorithm achieves higher accuracy
than the Metropolis algorithm for a fixed number of
tables, despite the bias introduced by using the MLE
to calculate the transition probabilities. This is not
surprising, because the effective sample size of tables
drawn by the Metropolis algorithm is far smaller than
the number of tables. See Table~\ref{tabl:2} for
an example.

\subsection{No-three-way interaction model of three-way table}
\label{subsect:3way}

We investigated the performance of the approximate direct sampling
algorithm by exemplifying the following no-three-way interaction
model of three-way contingency tables with fixed marginal sums for
$s=1,2,5$ and $10$:
\[
\begin{array}{ccc|c}
    u_{111}&u_{112}&u_{113}&3s\\
    u_{121}&u_{122}&u_{123}&3s\\
    u_{131}&u_{132}&u_{133}&3s\\
    \hline
    3s&3s&3s&9s 
\end{array} \quad
\begin{array}{ccc|c}
    u_{211}&u_{212}&u_{213}&3s\\
    u_{221}&u_{222}&u_{223}&3s\\
    u_{231}&u_{232}&u_{233}&3s\\
    \hline
    3s&3s&3s&9s 
\end{array} \quad
\begin{array}{ccc|c}
    2s&2s&2s&6s\\
    2s&2s&2s&6s\\
    2s&2s&2s&6s\\
    \hline
    6s&6s&6s&18s 
\end{array},
\]
where the vector
\begin{align*}
u=(&u_{111},u_{112},u_{113},u_{121},u_{122},u_{123},u_{131},u_{132},u_{133},\\
&u_{211},u_{212},u_{213},u_{221},u_{222},u_{223},u_{231},u_{232},u_{233})^\top.
\end{align*}
is random, the vector of the sufficient statistics is
\begin{align*}  
  b=(&
  u_{11\cdot},u_{12\cdot},u_{13\cdot},u_{21\cdot},u_{22\cdot},u_{23\cdot},
  u_{1\cdot1},u_{1\cdot2},u_{1\cdot3},u_{2\cdot1},u_{2\cdot2},u_{2\cdot3},\\
  &u_{\cdot11},u_{\cdot12},u_{\cdot13},u_{\cdot21},u_{\cdot22},u_{\cdot23},
  u_{\cdot31},u_{\cdot32},u_{\cdot33})^\top\\
  =(&
  3s,\ldots,3s,2s,\ldots,2s)^\top,
\end{align*}
and the configuration matrix is
\[
  A=\left(
  \begin{array}{c}
    E_6\otimes 1_3\\
    E_2\otimes 1_3\otimes E_3\\
    1_2\otimes E_9
  \end{array}\right).
\]
A Markov basis consists of \cite{DS98} 
\[
  \begin{bmatrix}1i_2i_3\\1j_2j_3\\2i_2j_3\\2j_2i_3\end{bmatrix}-
  \begin{bmatrix}1i_2j_3\\1j_2i_3\\2i_2i_3\\2j_2j_3\end{bmatrix},
    \quad
    \{(i_2,j_2,i_3,j_3)\in[3]^4: i_2<j_2, i_3<j_3\},
\]
which represents nine vectors with the components of $u$ in the first
bracket are $+1$, those in the seconds brackets are $-1$, and
the other components are 0, and six vectors with 12 non-zero
elements including
\[
  \begin{bmatrix}11i_1\\12i_2\\13i_3\\21j_1\\22j_2\\23j_3\end{bmatrix}-
  \begin{bmatrix}11j_1\\12j_2\\13j_3\\21i_1\\22i_2\\23i_3\end{bmatrix},
\]
where $(i_1,i_2,i_3)$ is a permutation of $(1,2,3)$ and determines
$(j_1,j_2,j_3)$ by the condition $u\in{\rm Ker}A\cap\mathbb{Z}^{18}$.
For example, if $(i_1,i_2,i_3)=(1,3,2)$, then
$(j_1,j_2,j_3)=(2,1,3)$.

First, we assessed the convergence of the generalized IPS.
Unlike the case of the two-way contingency tables in
Section~\ref{subsect:2way}, setting $\epsilon$ in equation
\eqref{conv} needed care. Proposition~\ref{prop:lattice}
guarantees that a sample path is on the Markov lattice,
even for the approximate direct sampling algorithm.
However, if the convergence was insufficient, a sample path
hit outside the Markov lattice because of the numerical error
in computing transition probabilities with the floating point
arithmetic. We confirmed that the generalized IPS always
converged if we took small $\epsilon$ but demanded many
iterations. In the following numerical experiments, we set
$\epsilon=0.005$ throughout. If we could not achieve convergence
by 1,000 iterations (less than 0.06\%), we discarded the sample
path and drew another contingency table.

Then, for $s=2,5$ and $10$, we generated a half million
contingency tables using the approximate direct sampling
algorithm and computed the empirical distribution $\hat{p}_{\rm D}$
of the chi-square values. For $s=1$, simple computation yields
the distribution of chi-square values \eqref{chi2}:
\begin{equation}\label{no3_e}
  p(0)=\frac{16}{37}\doteq0.432, \quad
  p(8)=\frac{18}{37}\doteq0.486, \quad
  p(12)=\frac{3}{37}\doteq0.081,
\end{equation}
and $p(z)=0$ if $z\notin\{0,8,12\}$, and we used this.

We assessed the convergence of the Metropolis chain to
the stationary distribution. For the Metropolis chain of four
pairs of the burn-in and length:
$(10^3,10^4)$, $(10^4,10^4)$, $(10^5,10^4)$ and $(10^5,10^5)$,
respectively, we computed the empirical distribution $\hat{p}_{\rm M}$
of the chi-square values and calculated the total variation
distance \eqref{TV}. Here, we set the expected value of each cell
to $s$. We show the averages of the results obtained by repeating
this procedure 10 times in Table~\ref{tabl:4}. For $s=1$, we have
\[
  \hat{p}_M(0)\doteq0.427, \quad
  \hat{p}_M(8)\doteq0.491, \quad
  \hat{p}_M(12)\doteq0.082, \quad
\]
and $p(z)=0$ if $z\notin\{0,8,12\}$. The closeness to
the true distribution \eqref{no3_e} of the chi-square values
demonstrates that the MLE performs a good approximation of
the UMVUE. The pairs $(10^4,10^4)$ and $(10^5,10^4)$ were
equally good, which implies that the Metropolis chain converged
after 100,000 steps. An important observation is that the total
variation distance of the pair $(10^5,10^5)$ decreases as $s$
increase except $s=1$ (the Monte Carlo error was small for $s=1$
since we used not the empirical distribution $\hat{p}_D$ but
the true distribution \eqref{no3_e}).
  This is plausible, since as increasing the sample size
  reduces the bias of the MLE (see Remark~\ref{rema:bias}).  
Unlike the case of the two-way contingency tables
(Table~\ref{tabl:1}), the bias decrease is evident.
\begin{table}[t]
  \small
  \caption{The total variation distances between the empirical
    distributions of chi-square values generated by
    the Metropolis algorithm and that by the approximate direct
    sampling algorithm.}
  \label{tabl:4}\smallskip
  \centering
  \begin{tabular}{ccccc}
       &\multicolumn{4}{c}{(burn-in, length)}\\
  $s$&$(10^3,10^4)$&$(10^4,10^4)$&$(10^5,10^4)$&$(10^5,10^5)$\\
  \hline  
  1  &0.019 &0.028 &0.023 &0.007\\ 
  2  &0.037 &0.027 &0.034 &0.029\\
  5  &0.030 &0.039 &0.036 &0.017\\
  10 &0.035 &0.031 &0.031 &0.013\\
  \hline
  \end{tabular}  
\end{table}

Given the above results, we drew 10,000 tables after
the burn-in of 100,000 steps for the Metropolis chain.
We computed the ESS of the 10,000 tables drawn by
the Metropolis algorithm and drew the same number of
tables using the direct sampling algorithm. Table~\ref{tabl:5}
shows the averages of the results obtained by repeating
this procedure 10 times. It shows the times it took
to draw 10,000 tables using the Metropolis algorithm
(including the burn-in), the time it took to draw the
tables of the number of the ESS using the approximate direct
sampling algorithm, and one table using the direct sampling
algorithm. The approximate direct sampling algorithm was
uniformly less efficient than the Metropolis algorithm.
When the total count is 180 ($s=10$), it took about 36 times
as long as the Metropolis algorithm.
\begin{table}[t]
    \small
    \caption{The times to draw 10,000 tables by the Metropolis
      algorithm, the effective sample sizes (ESS) of the 10,000
      tables, the times to draw the same number of tables by
      the approximate direct sampling algorithm, and the times
      per single table. Times are in seconds.}
  \label{tabl:5}\smallskip
    \centering
  \begin{tabular}{ccccc}
$s$ &ESS &Metropolis & Approx.-Direct & 1 table\\
\hline
1   &444 &3.877 & 25.4&0.057\\ 
2   &844 &4.012 & 83.7&0.099\\
5   &717 &4.241 &130.3&0.182\\
10  &544 &4.574 &165.0&0.303\\
\hline
\end{tabular}
\end{table}

\section{Discussion}\label{sect:disc}

Previous implementations of the direct sampling algorithm
\cite{Man17,MT21+} have enabled exact sampling, but depending
on the model, the computational burden makes it challenging to
put into practical use. In this study, we proposed
the approximate algorithm applicable to any log-affine model
with the expense of the bias, that asymptotically disappears.
In contrast to deriving the Markov basis in the Metropolis
algorithm or the connection matrix of the holonomic ideal in
the exact sampling, we can implement the approximate sampling
algorithm straightforwardly since it demands no prior
investigations; it only needs existing methods to compute the MLE.
Furthermore, as shown in Section~\ref{subsect:2way}
for a non-independence model of two-way contingency tables,
the approximate direct sampling algorithm achieves more accurate
sampling than the Metropolis algorithm for a fixed number of
tables in the same order of time, despite the bias coming from
using the MLE to compute the transition probability.

In the worst case, approximate direct sampling for
the no-three-way interaction model in Section~\ref{subsect:2way}
took 36 times the computational time of the Metropolis algorithm.
While it is true that approximate direct sampling is computationally
expensive, we should note that a Markov basis for the no-three-way
interaction model is known and given. If a Markov basis is unknown
for a model and cannot be calculated, the Metropolis algorithm
cannot be applied. In contrast, the approximate direct sampling is
always feasible as long as maximum likelihood estimation is
possible.

The approximate direct sampling algorithm can be improved.
In terms of accuracy, the MLE approaches the UMVUE by
reducing the bias. The UMVUE could be used in the later steps
of the algorithm where the total number of counts becomes small.
For computing the MLE, Newton's method is likely faster than
the generalized iterative proportional scaling.
However, these are typical tasks in statistical
analysis when applying an algorithm to each model and are beyond
the scope of this paper, which aims to propose an approximate
sampling algorithm in its most general form.

Sampling from discrete exponential families has been
a long-standing problem, but as this paper showed, we can relax
it by sequentially applying the MLE instead of the UMVUE. This
relaxation is a simple idea; it would not be surprising if it had
existed half a century ago. 

\section*{Acknowledgements}

The author thanks the referees for their valuable comments
on improving the presentation.
He also thanks a referee for pointing out the mistake in the
original version of Theorem~2.9 and suggesting improvements.
This work was supported
in part by JSPS KAKENHI Grant Numbers 20K03742 and 24K06876.

\section*{Appendix}

We show that the rational MLE and the UVMUE coincide
in the two-way contingency table of the quasi-independence
model that appeared in Example~\ref{exam:q2x2}. If we fix
$u_{13}$ and $u_{31}$, we have the $2\times 2$ subtable:
\[
\begin{array}{cc|c}
    u_{11}&u_{12}&u_{1\cdot}-u_{13}\\
    u_{21}&u_{22}&u_{2\cdot}-u_{23}\\
    \hline
    u_{\cdot1}-u_{31}&u_{\cdot2}-u_{32}&
    u_{1\cdot}+u_{2\cdot}-u_{\cdot3}
  \end{array},
\]
where
$u_{1\cdot}+u_{2\cdot}-u_{\cdot3}=u_{\cdot1}+u_{\cdot2}-u_{3\cdot}$,
$u_{3\cdot}=u_{31}+u_{32}$, and $u_{\cdot3}=u_{13}+u_{23}$.
The probability mass function of the hypergeometric distribution
gives
\[
\sum_{u_{11}}\frac{1}{u_{11}!u_{12}!u_{21}!u_{22}!}
=\frac{(u_{1\cdot}+u_{2\cdot}-u_{\cdot3})!}
{(u_{1\cdot}-u_{13})!(u_{2\cdot}-u_{23})!(u_{\cdot1}-u_{31})!
 (u_{\cdot2}-u_{32})!}.
\]
Therefore,
\begin{align*}
Z_A(b;1)&=\sum_{u_{13},u_{31}}
\frac{1}{u_{13}!u_{23}!u_{31}!u_{32}!}
\frac{(u_{1\cdot}+u_{2\cdot}-u_{\cdot3})!}
{(u_{1\cdot}-u_{13})!(u_{2\cdot}-u_{23})!(u_{\cdot1}-u_{31})!
  (u_{\cdot2}-u_{32})!}\\
&=(u_{1\cdot}+u_{2\cdot}-u_{\cdot3})!\sum_{u_{13}}\frac{1}{u_{13}!(u_{1\cdot}-u_{13})!(u_{\cdot3}-u_{13})!(u_{2\cdot}-u_{\cdot3}+u_{13})!}\\
& ~~ \times \sum_{u_{31}}\frac{1}{u_{31}!(u_{\cdot1}-u_{31})!(u_{3\cdot}-u_{31})!(u_{\cdot2}-u_{3\cdot}+u_{31})!}\\
&=\frac{(u_{1\cdot}+u_{2\cdot}-u_{\cdot3})!}
{u_{1\cdot}!u_{\cdot 3}!(u_{2\cdot}-u_{\cdot3})!
  u_{\cdot1}!u_{3 \cdot}!(u_{\cdot2}-u_{3\cdot})!}\\
& ~~ \times
{}_2F_1(-u_{1\cdot},-u_{\cdot 3},u_{2\cdot}-u_{\cdot3}+1)
\times{}_2F_1(-u_{\cdot1},-u_{3\cdot },u_{\cdot2}-u_{3\cdot}+1)\\
&=\frac{(u_{1\cdot}+u_{2\cdot})!(u_{\cdot1}+u_{\cdot2})!}
{u_{\cdot1}!u_{\cdot2}!u_{\cdot3}!u_{1\cdot}!u_{2\cdot}!u_{3\cdot}!(u_{\cdot1}+u_{\cdot2}-u_{3\cdot})!},
\end{align*}
where in the last equality we used the Gauss hypergeometric
theorem
\[
  {}_2F_1(\alpha,\beta;\gamma;1)=
  \frac{\Gamma(\gamma)\Gamma(\gamma-\alpha-\beta)}
       {\Gamma(\gamma-\alpha)\Gamma(\gamma-\beta)},
  \quad
  {\rm Re}(\gamma)>{\rm Re}(\alpha+\beta), \quad      
  -\gamma\not\in \mathbb{N}
\]
for $\alpha,\beta,\gamma\in\mathbb{C}$. Then, the UMVUE is
\[
  \tilde{\mu}_{13}(b;1)=\frac{Z_A(b-a_{13};1)}{Z_A(b;1)}=
  \frac{u_{1\cdot}u_{\cdot3}}{u_{1\cdot}+u_{2\cdot}},
\]
where $a_{13}$ is the column vector of $A$ corresponding
to the first row and the third column. The UMVUE coincides
with the rational MLE in Example~2.5 of \cite{CS21}.

\begin{flushleft}

Shuhei Mano\\
The Institute of Statistical Mathematics, Tokyo 190-8562, Japan\\
E-mail: smano@ism.ac.jp

\end{flushleft}

\end{document}